\documentclass[11pt]{amsart}

\usepackage[english]{babel}

\usepackage{amsmath,amssymb,amscd}

\usepackage{graphicx}
\usepackage{subcaption}

\usepackage{hyperref}
\usepackage{enumitem}

\title{Invariant random subgroups in hyperbolic reflection groups}

\author{Jean Raimbault}
\address{Institut de Mathématiques de Marseille, UMR 7373, CNRS, Aix-Marseille Université}
\email{jean.raimbault@univ-amu.fr}

\DeclareFontFamily{U}{wncy}{}
\DeclareFontShape{U}{wncy}{m}{n}{<->wncyr10}{}
\DeclareSymbolFont{mcy}{U}{wncy}{m}{n}
\DeclareMathSymbol{\Sha}{\mathord}{mcy}{"58}

\newcommand{\sub}{\mathrm{Sub}}
\newcommand{\eps}{\varepsilon}

\newcommand{\wdt}[1]{\widetilde #1}
\newcommand{\wdh}[1]{\widehat #1}

\newcommand{\pl}{\partial}
\newcommand{\bs}{\backslash}

\newcommand{\vol}{\operatorname{vol}}

\newcommand{\id}{\operatorname{Id}}

\newcommand{\interval}[4]{
  \ifthenelse{ \equal{#1}{o} } {\mathopen{]}} {\mathopen{[}}
  #2, #3
  \ifthenelse{ \equal{#4}{o} } {\mathclose{[}} {\mathclose{]}}
}

\newcommand{\PO}{\mathrm{PO}}

\newcommand{\PGL}{\mathrm{PGL}}

\newcommand{\Conv}{\mathrm{Conv}}

\newcommand{\supp}{\mathrm{supp}}

\newcommand{\RR}{\mathbb R}
\newcommand{\ZZ}{\mathbb Z}
\newcommand{\HH}{\mathbb H}
\newcommand{\PP}{\mathbb P}

\makeatletter
\DeclareFontFamily{OMX}{MnSymbolE}{}
\DeclareSymbolFont{MnLargeSymbols}{OMX}{MnSymbolE}{m}{n}
\SetSymbolFont{MnLargeSymbols}{bold}{OMX}{MnSymbolE}{b}{n}
\DeclareFontShape{OMX}{MnSymbolE}{m}{n}{
    <-6>  MnSymbolE5
   <6-7>  MnSymbolE6
   <7-8>  MnSymbolE7
   <8-9>  MnSymbolE8
   <9-10> MnSymbolE9
  <10-12> MnSymbolE10
  <12->   MnSymbolE12
}{}
\DeclareFontShape{OMX}{MnSymbolE}{b}{n}{
    <-6>  MnSymbolE-Bold5
   <6-7>  MnSymbolE-Bold6
   <7-8>  MnSymbolE-Bold7
   <8-9>  MnSymbolE-Bold8
   <9-10> MnSymbolE-Bold9
  <10-12> MnSymbolE-Bold10
  <12->   MnSymbolE-Bold12
}{}

\let\llangle\@undefined
\let\rrangle\@undefined
\DeclareMathDelimiter{\llangle}{\mathopen}%
                     {MnLargeSymbols}{'164}{MnLargeSymbols}{'164}
\DeclareMathDelimiter{\rrangle}{\mathclose}%
                     {MnLargeSymbols}{'171}{MnLargeSymbols}{'171}
\makeatother

\numberwithin{equation}{section}

\begin{document}

\newtheorem{theo}{Theorem}
\newtheorem{lem}[theo]{Lemma}
\newtheorem{prop}[theo]{Proposition}
\newtheorem{cor}[theo]{Corollary}
\newtheorem{question}{Question}

\begin{abstract}
  We prove that certain Fuchsian reflection groups admit invariant random subgroups having uncountably many isomorphism types of subgroups in their support, providing an answer to a question of S.~Thomas. We also give similar constructions in higher-dimensional spaces. Our constructions are based on limits of compact Coxeter polytopes in hyperbolic spaces. We also provide examples of invariant random subgroups related to questions of Y.~Glasner and A.~Hase through a similar construction.  
\end{abstract}

\maketitle

The notion of {\em invariant random subgroup} is a probabilistic generalisation of that of a normal subgroup. An invariant random subgroup in a discrete groups $\Gamma$ is just a Borel probability measure on the Chabauty space of subgroups of $\Gamma$, which is invariant under conjugation by elements of $\Gamma$. Fundamental examples are given by subgroups whose conjugacy class is finite, for instance finite-index subgroups and normal subgroups. The fact that the space of invariant random subgroups is compact and so that one can study families of finite-index subgroups through their limits was the original motivation for the introduction of these objects (see for instance \cite{AGV} where the terminology was first used). On the other hand, as the properties of normal subgroups in a given group can give insight into the structure of this group or its study up to isomorphism, so should that of invariant random subgroups. The present paper is largely motivated by the latter paradigm.  

More precisely, we are interested in groups which contain {\em diffuse} invariant random subgroups, a notion which was introduced by S.~Thomas in \cite{Thomas_diffuse}. These are invariant random subgroups such that for any group in their support, the measure of subgroups which are isomorphic to it is zero. It is not obvious that such IRSs exist in any group at all, but in loc. cit. Thomas constructs examples of discrete groups which admits diffuse ergodic IRS and asks whether there are more ``natural'' examples of such. Our goal here is mostly to provide such examples in a geometric context. We also use similar constructions to provide examples of faithful and spanning IRSs as asked in \cite{Glasner_Hase}.


\subsection*{Fuchsian groups}

In free or surface groups it is well-known how to construct ergodic IRS with uncountable support (see for example \cite{Bowen}) but all subgroups in their support are free groups of countable rank and so there is only one isomorphism class. On the other hand we show that it is possible to construct diffuse IRS in some Fuchsian groups with torsion. Specifically, we will do this for the group $\Gamma$ which is the (4,4,4)-triangle group. That is, let $T_4$ be the unique triangle in $\HH^2$ with three angles equal to $\pi/4$, and let $\Gamma_{T_4}$ be the subgroup of isometries generated by reflections in the faces of $T_4$. This is a discrete group of isometries and the translates of $T_4$ by elements of $\Gamma$ tile $\HH^2$. It has the simple presentation
\[
\Gamma_{T_4} = \left\langle s_1, s_2, s_3\: \middle| \: (s_is_{i+1})^4, s_i^2\right\rangle
\]
(where the indices are taken cyclically over $i=1, 2, 3$). With this notation we prove the following. 

\begin{theo} \label{main1}
  The group $\Gamma_{T_4}$ admits an ergodic diffuse IRS. 
\end{theo}

We expect that this is true of many reflection groups in $\HH^2$ (though not all, for instance right-angled groups and triangle groups with all angles equal to a given prime number have only countably many isomorphism types of subgroups, similar to free or surface groups) and that arguments similar to ours can be used in this setting, more generally it seems that the following question could have a definite answer.

\begin{question}
  Which finitely generated Fuchsian groups with torsion contain diffuse IRSs? 
\end{question}


\subsection*{Higher-dimensional right-angled Coxeter groups}

The question of whether lattices in higher-dimensional rank 1 groups\footnote{The Stuck--Zimmer theorem implies that lattices in higher-rank simple groups have no diffuse ergodic IRSs. } admit diffuse ergodic IRSs seems interesting and to require more sophisticated arguments. In this direction we prove the following theorem. 

\begin{theo} \label{main2}
  Let $P \subset \HH^d$, $d \ge 4$ be a right-angled polyhedron of finite volume. Then its reflection group $\Gamma_P$ admits diffuse ergodic IRSs.
\end{theo}

In dimension 3 we observe that our arguments give the weaker conclusion that there are IRSs of right-angled reflection groups whose support contains uncountably many isomorphism types. However, there exists ergodic IRSs which are not diffuse but satisfy the weaker property of having uncountably many isomorphism types in their support (Thomas constructed examples of such in the group of finitary permutations of a countable set). Thus we do not know if our construction for Theorem \ref{main2} yields ergodic diffuse IRSs for $d=3$, though we suspect it does. 

We note that it is well-known that there are infinitely many finite-volume right-angled polyhedra in $\HH^d$ for $2 \le d \le 8$ (which can be chosen to be compact for $d=2,3,4$) and none for $d \ge 13$ (\cite{PV,Dufour}), so Theorem \ref{main2} provides infinitely many exemples of lattices in $\PO(d, 1)$ for $4 \le d \le 8$ containing diffuse ergodic IRSs but cannot be applied in large dimensions. It seems plausible that the following question has a positive answer. 

\begin{question}
  Do all lattices in rank simple 1 Lie groups not isogenous to $\PGL_2(\RR)$ contain a diffuse ergodic IRS?
\end{question}

Note that there is (to my knowledge) no current example of a torsion-free hyperbolic group with a diffuse ergodic IRS.


\subsection*{Faithful and spanning IRSs}

We also construct examples of IRSs in right-angled groups with other properties. In \cite{Glasner_Hase} the authors introduce the notion of a faithful IRS and that of a spanning IRS (see \ref{def_GH} below). They prove that any Gromov-hyperbolic group admits a faithful ergodic IRS with uncountable support (in fact a more general result). On the other hand free groups contain ergodic IRSs which are both faitful and spanning (for instance the examples in \cite{Bowen}). We provide hyperbolic examples of in higher dimensions with the following result. 

\begin{theo} \label{faithful}
  Let $d \ge 2$ and let $P \subset \HH^d$ be a compact right-angled polyhedron. Then $\Gamma_P$ admits ergodic IRSs which are faithful and spanning. 
\end{theo}

As above, this implies that there are infinitely many lattices in $\PO(d, 1)$, $d=3, 4$, which contain faithful spanning IRSs. We expect a straightforward generalisation to finite-volume right-angled polyhedra (we currently use an argument with convex envelopes which applies only to compact polyhedra but using thick parts it should work for general finite-volume polyhedra). A more interesting generalisation would be to all hyperbolic right-angled Coxeter groups.

%
%
%

\subsection*{Organisation}

In the first section \ref{sec:shift_IRS} we give a general construction of reflexive IRSs in Coxeter groups, which is a variation on that given in \cite[Section 13]{7s2}, and give conditions for it to result in diffuse IRSs. We then apply this in section \ref{sec:fuchsian} to Fuchsian groups to prove Theorem \ref{main1} and in Section \ref{sec:right_angled} to right-angled groups to prove Theorem \ref{main2}. Finally, in section \ref{sec:faithful} we prove Theorem \ref{faithful}, recalling there the necessary preliminaries from \cite{Glasner_Hase}. 


\subsection*{Acknowledgments} Supported by the Agence nationale de la recherche through grant AGDE - ANR-20-CE40-0010-01.

I am grateful to Simon Thomas and Egemen \c Ceken for comments on a previous version of this paper. 


%
%

\section{Reflective IRS over a shift} \label{sec:shift_IRS}

\subsection{Construction}

Let $P$ be a finite Coxeter polyhedron. Assume that we can find a pair of finite Coxeter polyhedron $\wdh P_i$ both of which are tiled by $P$ (i.e. each $\Gamma_{\wdh P_i}$ is a finite-index subgroup of $\Gamma_P$) and for each $i$ faces $F_i^\pm$ of $\wdh P_i$ such that $F_1^\pm$ is a polyhedron isometric to $F_2^\mp$. We also ask that the dihedral angles with the faces adjacent to $F_i^{\pm}$ are of the form $\tfrac\pi{2m}$ so that gluing the $\wdh P_i$ along these faces results in a Coxeter polyhedron.

Then for a sequence $a = (a_n)_{n\in \ZZ} \in \{1, 2\}^\ZZ$ we can form the infinite polyhedron $P_a$ obtained by gluing the $F_{a_i}^+$-face of $\wdh P_{a_i}$ to the $F_{a_{i+1}}^-$-face of $\wdh P_{a_{i+1}}$ for all $i \in \ZZ$.

For a subgroup $H$ of $\Gamma$ we denote by $[H]$ its conjugacy class ; we note that the conjugacy class $[\Gamma_{P_a}]$ is well-defined. We have the following observation. 

\begin{lem} \label{chabclosedPa}
  If $A$ is a shift-invariant closed subset of $\{1, 2\}^\ZZ$ then the subset $\bigcup_{a \in A} [\Gamma_{P_a}]$ is Chabauty-closed. 
\end{lem}

Let $n_i = [\Gamma_P : \Gamma_{P_i}]$. Let $\nu$ be a shift-invariant Borel probability measure on $\{1, 2\}^\ZZ$, let
\[
t_\nu = \PP_\nu(a_0=1)
\]
and define another probability measure $\nu'$ as follows : 
\[
\nu'(A) =  \frac1{t_\nu n_1+(1-t_\nu)n_2}\int_A n_{a_0}d\nu(a). 
\]
Now we define a random subgroup $\mu_n\nu$ of $\Gamma_P$ as follows: let $a$ be a $\nu'$-random sequence. The polyhedron $P_a$ is $P$-tiled ; we choose a uniformly random copy of $P$ inside the finite polyhedron $P_{a_0}$ and choose the corresponding conjugate of $\Gamma_{P_a}$ in $\Gamma_P$ as our random subgroup.

\begin{lem}
  Let $\nu$ be a shift-invariant Borel probability measure on $\{1, 2\}^\ZZ$. The random subgroup $\mu_\nu$ is invariant. 
\end{lem}

\begin{proof}
  The simplest argument is by showing that the measure $\mu_\nu$ is a weak limit of finitely-supported IRSs. 

  Assume first that $\nu$ is periodic and has period a finite sequence $b = (b_1, \ldots, b_n)$, the orbifold $P_\alpha$ is the finitely-supported IRS corresponding to the $D_\infty$-cover of $\wdh P_b$ obtained by sending the two reflections associated with the faces $F_{b_1}^-$, $F_{b_n}^+$ to the generators of $D_\infty$ and the other reflection generators of $\Gamma_{\wdh P_b}$ to the identity. Let $a = (\cdots b,b,\cdots)$, then the normaliser of $\Gamma_{P_a}$ in $\Gamma_P$ contains $\Gamma_{\wdh P_b}$ so that we can choose a random conjugate of $\Gamma_{P_a}$ by picking a uniformly random element of $\Gamma_P/\Gamma_{\wdh P_b}$. Geometrically this means uniformly picking a random translate of $P$ in $\wdh P_b$, which amounts to picking $1 \le i \le n$ according to the distribution with weight $\frac{n_i}{n_1 + \cdots +n_n}$ at $i$, and then picking a random translate of $P$ inside $P_{b_i}$. This is exactly the definition of $\mu_\nu$ in this case, hence it is the uniform distribution on the conjugacy class of $\Gamma_{P_b}$. 

  In general, the measure $\nu$ itself is a weak limit of periodic measures $\nu_n$. It is immediate to verify that the measure $\mu_\nu$ is then the weak limit of the IRSs $\mu_{\nu_n}$, hence it is invariant itself. 

  For a more direct argument which can be adapted to our situation see \cite[Lemma 13.4]{7s2}. 
\end{proof}

\begin{lem}
  If $\nu$ is shift-ergodic then $\mu_\nu$ is an ergodic IRS.
\end{lem}

\begin{proof}
  
  Let $\wdt A$ be a conjugacy-invariant Borel subset in the support of $\mu_\nu$ and suppose that $\mu_\nu(\wdt A) > 0$. Then by Lemma \ref{chabclosedPa} all subgroups in $\wdt A$ are of the form $\Gamma_{P_a}$ for some $a$ in the support of $\nu$. Let $A = \{ a : [\Gamma_{P_a}] \subset \wdt A \}$. Then $A$ is shift-invariant and Borel. If $\nu(A) = 0$ then $\mu_\nu(\wdt A) = 0$ as cen be seen from the definition of $\mu_\nu$. So we must have $\nu(A) = 1$, and it follows that $\wdt A$ has full measure. 
\end{proof}


\subsection{Rigidity conditions for diffuseness}

We fix a Coxeter polyhedron $P$ in $\HH^d$ as in the previous section. 

\begin{lem} \label{criterion}
  If :
  \begin{enumerate}
  \item \label{rigidity_Pa} the polyhedra $P_a$ for $a \in \supp(\nu)$ satisfy that $\Gamma_{P_a}$ is isomorphic to $\Gamma_{P_{a'}}$ if and only if $P_a$ is isometric to $P_{a'}$ ;

  \item \label{recognition} the normaliser of $\Gamma_{P_a}$ in $\Gamma_P$ is finite unless $a$ is shift-periodic 
  \end{enumerate}
  and $\nu$ is a shift-invariant ergodic measure on $\{1, 2\}^\ZZ$ which is not supported on a finite shift-orbit then $\mu_{\nu}$ is diffuse. 
\end{lem}

\begin{proof}
  The support of $\mu_\nu$ is contained in $\{\Gamma_{P_a} :\: a \in \supp(\nu)\}$. Hence by condition \ref{rigidity_Pa} each isomorphism class in $\supp(\mu_\nu)$ is contained in a $\PO(d, 1)$-conjugacy class. By Lemma \ref{countable2} below this class splits into at most countably many $\Gamma$-conjugacy classes. Hence, if an isomorphism class had positive $\mu_\nu$-measure then some $\Gamma$-conjugacy class would as well. By ergodicity of $\mu_\nu$, it follows that $\mu_\nu$ would have to be supported on a single finite conjugacy class. Finally condition \ref{recognition} implies that this can happen if and only if $\nu$ is supported on a finite orbit. 
\end{proof}


\begin{lem} \label{countable2}
  Let $G$ be a semisimple Lie group and $\Gamma, \Lambda \subset G$ a pair of Zariski-dense discrete subgroups of $G$. There are at most countably many elements of $G$ conjugating $\Lambda$ into $\Gamma$. 
\end{lem}

\begin{proof}
  For a given element $g \in G$ the set of elements conjugating $g$ into $\Lambda$ is a countable (possibly empty) union of $\Lambda$-cosets of the centraliser $Z_G(g)$. Since $\Lambda$ is Zariski-dense we have that $\bigcap_{g \in \Lambda} Z_G(g) = Z(G)$ is discrete, so the set of elements conjugating every element of $\Lambda$ into $\Lambda$ is 0-dimensional, in particular countable.
\end{proof}


\section{Diffuse IRSs in Fuchsian groups} \label{sec:fuchsian}

Let :
\begin{itemize}
\item $T$ be the equilateral hyperbolic triangle with angles $\tfrac\pi 4$ ;

\item $O$ the regular hyperbolic right-angled octagon (which is tiled by 8 copies of $T$, see figure \ref{octagon} below) ;

\item $Q$ the half-octagon (see figure \ref{half_octagon}).  
\end{itemize}

Let $\wdh P_1 = O$, and $\wdh P_2$ the polygon obtained by gluing a copy of $Q$ to the a face of $O$ along a face adjacent to two right angles, and copies of $O$ of two opposite faces disjoint from the glued $Q$ (there is only one way of doing so up to isometry, see figure \ref{P2_fig}).

\begin{figure}[h]
  \begin{subfigure}{.3\textwidth}
    \includegraphics[width=\textwidth]{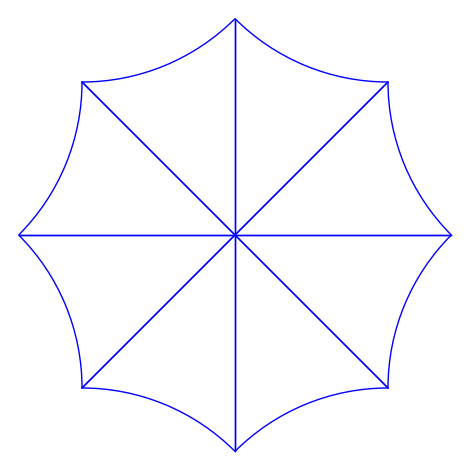}
    \caption{Regular octogon tesselated by triangles}
    \label{octagon}
  \end{subfigure}
  \hspace{1cm}
  \begin{subfigure}{.3\textwidth}
    \includegraphics[width=\textwidth]{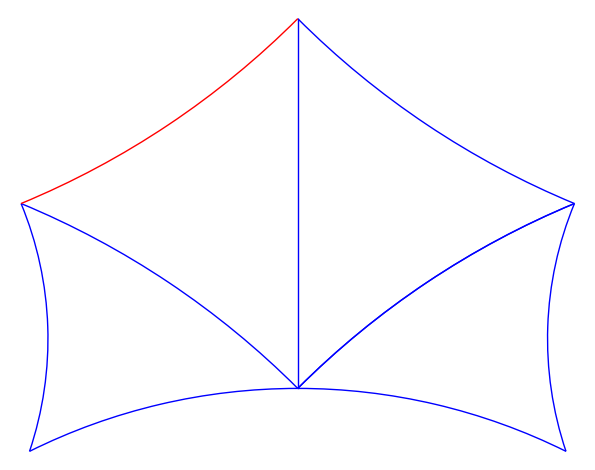}
    \caption{Half-octagon tesselated by triangles}
    \label{half_octagon}
  \end{subfigure}
\end{figure}

\begin{figure}[h]
  \includegraphics[width=.5\textwidth]{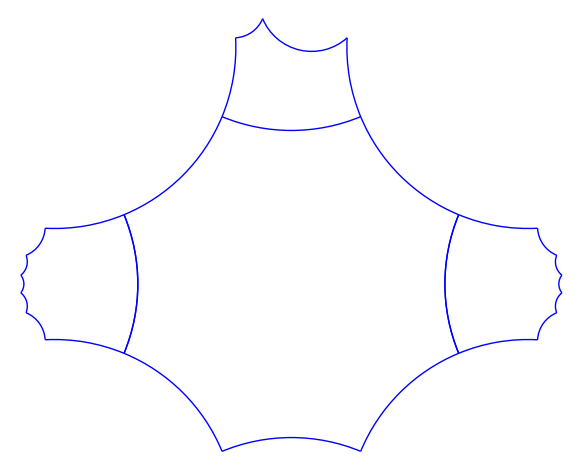}
  \caption{The polygon $\wdh P_2$}
  \label{P2_fig}
\end{figure}



\subsection{Rigidity}

The following lemma implies \eqref{recognition} from Lemma \ref{criterion} in our case. 

\begin{lem} \label{recog_fuchsian}
  With $\wdh P_1,\, \wdh P_2$ as above, $P_a$ admits a non-trivial self-isometry if and only if $a$ is shift-periodic.  
\end{lem}

\begin{proof}
  We may assume first that $a$ is not constantly 1 and let $l$ such that $a_l = 2$. Let $f \not= \id$ be a self-isometry of $P_a$. Let $x_1, x_2$ be the two vertices of $P_{a_l}$ with angle $\pi/4$. Then their images are consecutive vertices with angle $\pi/4$ on the boundary of $P_a$, hence they must belong to a single $P_{a_k}$ with $a_k = 2$. The polygon $P_2$ itself has no nontrivial self-isometry\footnote{Such an isometry would exchange the two vertices with angle $\tfrac \pi 4$ but their adjacent sides are of differet lengths. } so there is a unique isometry from $P_{a_l}$ to $P_{a_k}$. 

  The only isometries of $\HH^2$ sending a neighbourhood of each $x_i$ in $P_{a_l}$ to a neighbourhood of itself or the other $x_j$ is the identity or the reflection in the median between the two. The latter does not map $P_{a_l}$ inside $P_a$, hence we must have $f = f_0$. if $k=l$ it follows that $f=\id$, so we have $k \not= l$. 

  Now we prove that $a_{l+1} = a_{k+1}$ : assume that one of the two is equal to $2$ (otherwise they are both equal to 1 and we are finished). If $a_{l+1} = 2$ then it follows from the previous paragraph that $P_{a_{k+1}}$ and $f(P_{a_{l+1}})$ are both glued to the same edge of $P_{a_k}$. So after the first copy of $O$ there is are angles $\tfrac \pi 2$ and $\tfrac \pi 4$, which would be impossible if $a_{k+1} = 1$ (there would be at least 3 angles $\tfrac \pi 2$). Hence $a_{k+1} = 1$, and if we had started with $a_{k+1} = 1$ we would have proven in the same way that $a_{l+1} = 1$. Then a recursion argument shows that we must have $a_{n+k-l} = a_n$ for all $n$, that is $a$ has period $k-l \not= 0$. 
\end{proof}

If $P$ is a convex polygon in $\HH^2$ rooted at a point $x_0 \in P$ its angle sequences are the sequences in $[0,\pi[^\ZZ$ obtained by listing the angles at its vertices starting from the projections of the root $x_0$ on each boundary component, in counter-clockwise order (if $P$ is compact this is a single periodic sequence).

The following two lemmas show that the property \eqref{rigidity_Pa} in Lemma \ref{criterion} holds for the IRSs $\mu_\nu$ constructed from $\wdh P_1$ and $\wdh P_2$ as above. Together with Lemma \ref{recog_fuchsian} implying \eqref{recognition}, this shows that for a non-periodic $\nu$ the IRS $\mu_\nu$ provides an example proving Theorem \ref{main1}. 

\begin{lem}
  With $\wdh P_1,\, \wdh P_2$ as above, two polyhedra $P_a, P_{a'}$ are isometric to each other if and only if they have the same angle sequences up to shift. 
\end{lem}

\begin{proof}
  We show that the angle sequence determines $a$ up to shift. Now we may assume assume that the angle sequence contains a $(4,4)$ segment, and by shifting that $a_0 = 2$ and this segment corresponds to the two $\tfrac\pi 4$ angles in $P_{a_0}$. Let $N_1$ be the number of 2 in the angle sequence following $(4,4)$ and $N_{-1}$ the number of 2 preceding it. Then $k_1 = \tfrac{N_1-5}2$ is equal to the number of $i\ge 1$ such that $a_1 = \cdots = a_i = 1$, and $k_{-1} = \tfrac{N_{-1}-5}2$ the number of $i \le -1$ such that $a_{-1} = \cdots = a_i = 1$. Hence the angle sequence determines the number of $1$ in $a$ following and preceding a $2$, so it determines $a$ itself. 
\end{proof}

\begin{lem} \label{isomorphism}
  Let $\Gamma_1, \Gamma_2$ be reflection groups associated with polygons $P_1, P_2$ in $\HH^2$ all of whose angles are $\pi/2$ or $\pi/4$ and which have finitely many boundary components. If $\Gamma_1$ is isomorphic to $\Gamma_2$ then the angle sequences of $P_1$ and $P_2$ are shifted or mirrored from each other. 
\end{lem}

\begin{proof}
  We note that a much more general result is given in \cite[Theorem 4.2]{Bahls_book} (the result there is only written for finitely generated Coxeter groups but it generalises to Coxeter groups associated to locally finite graphs). We give a proof in our particular case where it is much simpler than in general dur to the linear nature of the defining graphs. 

  \medskip

  We first prove the case where $P_i$ have a single boundary component. 

  We will denote the angle sequences of the polyhedra by $(\alpha_i)$ and $(\beta_i)$ and use the canonical presentations
  \[
  \Gamma_1 = \langle s_i | (s_is_{i+1})^{\pi/\alpha_i}, s_i^2 \rangle, \: \Gamma_2 = \langle t_i | (t_it_{i+1})^{\pi/\beta_i}, t_i^2 \rangle.
  \]
  Let $\varphi : \Gamma_1 \to \Gamma_2$ be an isomorphism; we will prove that for every $i \in \ZZ$, $\varphi(s_i)$ is conjugated to $t_{\eps i+k}$ for some $\eps \in \{\pm 1\}$ and $k \in \ZZ$, from which the claim about the angle sequences follows immediately. 

  Any involution in $\Gamma_1$ is either a reflection or the product of two reflections with orthogonal axes, and in the latter case the pair of reflections has to be conjugate to a pair $s_i, s_{i+1}^g$ (for some $g \in \langle s_i, s_{i+1}\rangle$). We may assume that at least one of the $\alpha_i$ is equal to $\pi/2$, and shift the sequence for it to be $\alpha_0$, so that $\langle s_0, s_1\rangle \cong (\ZZ/2)^2$. All subgroups of $\Gamma_2$ isomorphic to $(\ZZ/2)^2$ are contained in one of the finite dihedral subgroups $\langle t_i, t_{i+1}\rangle$ so we must have that
  \[
  \varphi(s_0), \varphi(s_1) \in \langle t_k, t_{k+1}\rangle
  \]
  for some $k \in \ZZ$. Since $(s_0s_1)$ has no square root in $\Gamma_1$ we must have $\beta_k = \pi/2$, hence the involutions in $\langle t_k, t_{k+1}\rangle$ are exactly $t_k, t_{k+1}, (t_kt_{k+1})$. Moreover we cannot have $\varphi(s_0) = t_kt_{k+1}$: if this were the case then $\varphi(s_0), \varphi(s_{-1})$ would generate an infinite subgroup (since the fixed point sets of $\varphi(s_0)$ and $\varphi(s_{-1})$ are disjoint in this case) which is not possible since $\langle s_{-1}, s_0 \rangle$ is finite.

  It follows from the preceding paragraph that we have $\varphi(s_i) = t_k$ or $t_{k+1}$ and similarly $\varphi(s_1)$ is equal to the other among $t_k, t_{k+1}$. Reversing the order of the sequence $(\beta_i)$ if necessary we may assume that
  \[
  \varphi(s_0) = t_k, \: \varphi(s_1) = t_{k+1}.
  \]
  In the rest of this section we prove that $\varphi(s_2)$ is conjugated to $t_{k+2}$, and by using the argument as an induction step it follows that $\varphi(s_i)$ is conjugated to $t_{k+i}$ for all $i$.

  If $\alpha_1 = \pi/2$ then the same argument as above works to prove that $\varphi(s_2) = t_{k+2}$. If $\alpha = \pi/4$ then by the same argument as above $\varphi(s_2)$ is a reflection in $\langle t_{k+1}, t_{k+2} \rangle$ conjugated to $t_{k+2}$, that is either $t_{k+2}$ or $t_{k+1}t_{k+2}t_{k+1}$, since these are the only possibilities for an involution $t \in \langle t_{k+1}, t_{k+2} \rangle$ so that $(t_{k+1}t)^4 = 1$ but $(t_{k+1}t)^2 \not= 1$. 

  \medskip

  In the general case where $P_i$ may each have (finitely many) distinct boundary components, the groups $\Gamma_{P_i}$ are isomorphic to free products $\Gamma_{Q_{i, 1}}\ast\cdots\ast\Gamma_{Q_{i,m_i}}$ where $Q_{i, j}$ are Coxeter polyhedra corresponding to the boundary components of $P_i$. We claim that the reflection group of a Coxeter polyhedron with a single boundary component is freely indecomposable: more generally, a Coxeter group with connected diagram (where two generators are linked by an edge whenever their product has finite order) is so. To prove this note that, since they are involutions, all Coxeter generators must lie in a conjugate of some free factor in any decomposition of $\Gamma$ as a free product. On the other hand two generators connected by an edge in $G$ must lie in the same conjugate of the same free factor since their product has finite order, so by connectedness any two generators must lie in the same conjugate of the same free factor. It follows that the decomposition is trivial.

  By the Kurosh isomorphism theorem \cite{Kurosh_iso} it follows that the $\Gamma_{Q_{i, j}}$ are pairwise isomorphic, and by the previous case that their angle sequences must be pairwise equal or mirrored. 
\end{proof}


\section{Diffuse IRSs in right-angled reflection groups} \label{sec:right_angled}

In this section we assume $d \ge 3$. 

\subsection{Construction}

Let $P$ be a right-angled polyhedra of finite volume in $\HH^d$. We note that if $F, E$ are faces of $P$ which are not adjacent then they are disjoint (for instance because $P$ being right-angled, the link of every vertex is a simplex). 

Choose a face $F$ of $P$. Let $F'$ be a $F$-tiled polyhedron (in $\HH^{d-1}$) such that its number of $(m-2)$-faces is larger than that of any $K$-tiled polhedron whose dual graph is of diameter at most $d$ where $K$ is any face of $P$: since there are finitely many such possible graphs such a $F'$ exists, in fact any $F$-tiled polyhedron of sufficiently large volume works. We also require that the graph of the tiling of any $(m-2)$-face of $F'$ by faces of $F$ is of diameter at most 2 (we can ensure this for exemple by constructing $F'$ by doubling successively with respect to disjoint faces). 

Viewing $F'$ in the $P$-tiling of $\HH^d$ let $Q_0$ be the union of all copies of $P$ which intersect $F'$ in a translate of $F$. Let $Q$ be the subpolyhedron including only those $P$s lying on an arbitrarily chosen side of $F'$. Then $Q$ has a side isometric to $F'$.

We claim that all other faces of $Q$ have a tiling by some face of $P$, whose dual graph has diameter at most $d$. Let $D$ be such a face. We prove the claim first in the case where $D$ is not adjacent to $F'$: then $D$ is constructed from a face $E$ of $P$ at distance at least 2 from $F$ in the dual graph of $P$, by taking its reflections with respect to faces $F_0$ of $P$ which are adjacent to itself and to $F$ (we will say that $F_0$ is joining $E$ to $F$). If the distance from $E$ to $F$ is equal to 2 then $D$ is equal to $\Pi \cdot E$ where $\Pi$ is the stabiliser of the intersection of all faces joining $E$ to $F$, hence its dual graph has diameter $k$ where $k$ is the codimension in $\HH^d$ of the intersection, which is at most $(d-1)$. If the distance from $E$ to $F$ is $>2$, i.e. there are no faces joining $E$ to $F$, then the face $D$ is actually a face of $P$. 

In the case where $D$ is adjacent to $F'$, its tiling is isomorphic to that of its intersection with $F'$, which is of diameter at most 2 by the second hypothesis on $F'$ in the second paragraph. 

\medskip

Let $F''$ be a face at distance $>2$ from $F$ in the dual graph of $P$, and $F_+, F_-$ two faces of $P$ at distance $>2$ from $F''$ in the dual graph (we can ensure their existence by replacing $P$ by a suitable doubling). By gluing $P$ with $Q$ along a copy of $F''$ in $\pl Q$, we get a polyhedron $\wdh P_2$ which has one face isometric to $F'$ and all its other faces have strictly smaller valency (since they are tiled by faces of $P$ with graph of diameter at most $(d+1)$ by the claim above). Moreover all faces of $\wdh P_2$ adjacent to $F_+$ or $F_-$ are isometric to faces of $P$.

We further choose the gluing copy of $F''$ so that $\wdh P_2$ has no non-trivial self-isometry: this is possible since such an isometry must preserve both $F'$ (as it is the only highest degree face) and the glued $P$ (as it is the only polyhedron disjoint from $F'$), hence choosing the gluing so that the projection of the glued $P$ does not contain the center of $F''$ ensures that the only isometry is the identity.  

For the other polyhedron we simply choose $\wdh P_1 = P$. 


\subsection{Verification}

The following two lemmas establish that both conditions in Lemma \ref{criterion} are satisfied by the two polyhedra constructed above, proving Theorem \ref{main2}.  

\begin{lem}
  With $\wdh P_1,\, \wdh P_2$ as above, two polyhedra $P_a, P_{a'}$ are isometric to each other if and only if they $a=a'$ up to shift. 
\end{lem}

\begin{proof}
  First, it follows from the construction of $\wdh P_2$ that for any $a \in \{1, 2\}^\ZZ$, any face of maximal degree in $P_a$ is a $F'$-face in a copy of $\wdh P_2$. (this is true in $\wdh P_2$, and by gluing a copy of $\wdh P_1$ or $\wdh P_2$ along $F_+$ we only add faces with a tiling by faces of $P$ of diameter 2 since all faces of either $\wdh P_i$ adjacent to $F_\pm$ are faces of $P$). 

  Let $a \in \{1, 2\}^\ZZ$, let $k$ such that $a_k = 2$ and let $f$ be a self-isometry of $P_a$. Then $f(P_{a_k}) = P_{a_l}$ for some $l$ with $a_l = 2$ : by maximality of degree, $f$ has to map the $F'$-face $F_{a_k}'$ of $P_{a_k}$ to the $F'$-face $F_{a_l}'$ of some such $P_{a_l}$. Then, by the same argument as used for proving that $\wdh P_2$ has no symmetries we see that there is a unique isometry satisfying this which also maps the copy of $P$ disjoint from $F'$ in $P_{a_k}$ to the one in $P_{a_l}$, and as the latter is is the only copy of $P$ at distance 1 from $F'$ in $P_a$, $f$ must be equal to this isometry and in particular $f(P_{a_k} = P_{a_l}$ and $f$ is the only such isometry. By an inductive argument (see the proof of Lemma \ref{recog_fuchsian}) it follows that either $f$ is the identity or $k \not=l$ and $a$ is $|k-l|$-periodic. 
\end{proof}

\begin{lem}
  Let $d \ge 4$ and $R, R'$ be two right-angled polyhedra in $\HH^d$ all of whose $(d-1)$-faces are of finite $(d-1)$-volume and whose boundary is connected. Then $\Gamma_R$ is isomorphic to $\Gamma_{R'}$ if and only if $R$ is isometric to $R'$. 
\end{lem}

\begin{proof}
  By a particular case of the main result of \cite{Radcliffe_rigid}\footnote{see also \cite[Theorem 4.1]{Bahls_book}; a less general result which already includes right-angled Coxeter groups was proven earlier in the unpublished thesis of E.~Green. The statement in loc. cit. includes only right-angled Coxeter groups associated with finite graphs but the arguments there generalise with essentially no modification needed to groups associated with locally finite graphs. } any group isomorphism from $\Gamma_R$ to $\Gamma_{R'}$ is induced by an isomorphism between the dual graphs of $R$ and $R'$. Thus they are isomorphic as abstract polyhedra. It follows that their codimension 1 faces are also pairwise isometric to each other. Since these are compact hyperbolic right-angled polyhedra of dimension $d-1 \ge 3$, by the Mostow--Prasad strong rigidity theorem \cite{Mostow_rigidity,Prasad_rigidity} it follows that they must be pairwise isometric to each other. Since there is a unique way up to isometry to assemble two $(d-1)$-polyhedra at right angles along a codimension 1 face in $\HH^d$, it further follows that the boundary components of $P$ and $P'$ are pairwise isomorphic. Since $\pl P, \pl P'$ are connected there is a unique isometry of $\HH^d$ which extends the isometry between them, and which induces an isometry between $P$ and $P'$. 

\end{proof}


\section{Faithful and spanning invariant random subgroups in higher dimensions} \label{sec:faithful}

\subsection{Nomenclature} \label{def_GH}

Let $\Gamma$ be a countable group and $\mu$ an IRS in $\Gamma$. We recall two natural objects associated to $\mu$, which were introduced in \cite{Glasner_Hase}. The {\em kernel} $K_\mu$ of $\mu$ is the largest normal subgroup of $\Gamma$ which is contained in $\mu$-almost every subgroup (it is the generalisation to IRSs of the normal core of finite-index or almost-normal subgroups). In other words
\[
K_\mu = \{x\in\Gamma :\: \PP_\mu(x \in H) = 1\}. 
\]
The IRS $\mu$ is said to be {\em faithful} if $K_\mu = \{1\}$ (in other words there is no nontrivial normal subgroup contained in $\mu$-almost every subgroup). Concretely this means that
\[
\forall g \in \Gamma \setminus \{1\} \quad \PP_\mu(g \not\in H) > 0.  
\]

The {\em normal envelope} $E_\mu$ of $\mu$ is the smallest normal subgroups which contains $\mu$-almost every subgroup (generalisation of the subgroup normally generated), so that
\[
E_\mu = \langle x \in \Gamma :\: \PP_\mu(x \in H) > 0\rangle.
\]
The IRS $\mu$ is said to be {\em spanning} if $E_\mu = \Gamma$. Equivalently, there exists a normally generating set $S$ for $\Gamma$ such that $\PP_\mu(s \in H) > 0$ for all $s \in S$. 

\subsection{Convex envelopes}

If $P \subset \HH^d$ is a {\em compact} right-angled polytope and $S \subset \HH^d$, $d=3, 4$, define its convex envelope $\Conv_P(S)$ to be the smallest convex union of all polyhedra which contains $S$ in the tiling $\HH^d = \Gamma_P \cdot P$. The convex envelope is always a right-angled polyhedron. 

The following lemma generalises \cite[Lemmas 3.1 and 4.2]{Patel} from geodesic lines to convex subsets but its proof is almost immediate from the statements in loc. cit. 

\begin{lem} \label{envelope}
  If $C \subset \HH^d$ is convex then $\Conv_P(S)$ is contained in the $D+\log(2 + \sqrt 3)$-neighbourhood of $C$ (where $D$ is the diameter of $P$).
\end{lem}

\begin{proof}
  The statements of \cite[Lemma 3.1, Lemma 4.2]{Patel} give that for any geodesic line $\alpha \subset \HH^d$ we have $\Conv_P(\alpha) \subset N_R(\alpha)$ with $R = D+\log(2 + \sqrt 3)$, where $N_R(S)$ denotes the $R$-neighbourhood of $S$ in $\HH^d$. If $H$ is an hyperplane in $\HH^d$ then $\Conv_P(H) = \bigcup_{\alpha \subset H} \Conv_P(\alpha)$ and it follows immediately that $\Conv_P(H) \subset N_R(H)$. It follows immediately that the same is true for a half-space $H^+$ delimited by $H$. 

  Now if $C$ is an arbitary convex subset we have
  \[
  \Conv_P(C) = \bigcap_{H^+ \supset C} \Conv_P(H^+)
  \]
  (immediate since $C = \bigcap_{H^+ \supset C} H^+$) so that
  \[
  \Conv_P(C) = \bigcap_{H^+ \supset C} N_R(H^+).
  \]
  As $\bigcap_{H^+ \supset C} N_R(H^+) = N_R(C)$ (which can be seen in any CAT(0)-space by using projections on $C$) we get the result for $C$ as well. 
\end{proof}


\subsection{Proof of Theorem \ref{faithful}}

Le $P$ be a compact right-angled polyhedron in $\HH^d$, $d \ge 2$, and $\Gamma=\Gamma_P$. We will use the notation
\[
B_R = \{\gamma\in\Gamma :\: d(o, \gamma o) < R\}. 
\]
We fix an interior point $o$ in $P$ and let $P_n = \Conv_P(B(o, n))$ and $\mu_n$ be the IRS $\mu_{\Gamma_{P_n}}$ in $\Gamma$. 

Let $\mu_0$ be an accumulation point of $\mu_n$. Write $\mu_0 = t\mu_1 + (1-t)\delta_{\{\id\}}$; by \cite{BSlim_coxeter} we have $t>0$ and $\mu_1$ is supported on reflection groups: in fact it follows from Lemma \ref{equidis} below we have $t = 1$ in this case. We will prove that almost-every ergodic component of $\mu_1$ must be faithful and spanning.


\subsubsection{Spanning-ness}

\begin{lem} \label{equidis}
  There are constants $C, c > 0$ such that for all $n, R$ and all face types $s$ of $P$
  \[
  \frac{\vol(x\in P_n : \text{ there exists a face of type $s$ at distance $\le R$ from } x)}{\vol(P_n)} \ge 1-Ce^{-cR}.
  \]
\end{lem}

\begin{proof}
  We will prove the following statement as a corollary of Ratner's classification of unipotent-invariant measures: there exists $R > 0$ such that for all $n$, for any $x$ on the sphere $\pl B(o, n)$ there is a face of type $s$ of $P_n$ at distance $\le R$ from $x$. The lemma follows immediately since it follows from Lemma \ref{envelope} and the exponential growth of the volume of hyperbolic balls that the proportion of volume of $P_n$ at distance $\ge T$ from $\pl B(o, n)$ is $O(e^{-dT})$ for some $d>0$, uniformly in $n$. 

  To prove this claim we note that if $v \in T^1P$ is an outwardly pointing normal vector to a face $F$ of a translate $\gamma P$ (for some $\gamma \in \Gamma_P$) and to $\partial B(o, n)$ at the same time then $F$ lies on the boundary of $P_n$. This remains true for $v$ in an open neighbourhood in the interior of $P_n$ (and still tangent to $B(o, n)$). Let $U_s$ be such a neighbourhood translated back to $P$: we may assume that baspoints of vectors in $U_s$ lie at distance $\le 1$ from $\pl P_n$. 

  Rooted at a boundary point, the outward normal field to $\pl B(o, n)$ converges (in pointed Gromov--Hausdorff topology on subsets of $T^1\HH^d$) to the outward normal field to an horosphere. It we consider the uniform probability measure on the image in $T^1P = \Gamma_P \bs \HH^d$ of a $R$-ball in the latter, it follows from amenability of unipotent subgroups and Ratner's measure classification theorem \cite{Ratner} that this measure converges weakly to the Liouville measure on $T^1P$. In particular, for $R$ large enough the probability of lying in each $U_s$ is positive, so (for $n$ large enough, a priori depending on $R$) it is also positive for the $R$-ball on $\pl B(o, n)$ (rooted at any point). By the preceding paragraph this means that there is a face of type $s$ of $P_n$ at distance $\le R+1$ from any point on $\pl B(o, n)$.
\end{proof}

It follows from this lemma that for any $R>0$, if we enumerate $B_R = \{g_1, \ldots, g_{m_R} \}$ then for any Coxeter generator $s$ of $\Gamma$
\[
\PP_{\mu_n}(g_isg_i^{-1} \in H \text{ for some } 1 \le i \le m_R) \ge 1 - Ce^{-cR}.
\]
It follows that 
\[
\PP_{\mu_n}(\forall s \, \exists g_s \in B_R :\: g_ssg_s^{-1} \in H) \ge 1-C'e^{-cR}.
\]
Since $B_R$ is finite it follows that the same property is true for $\mu_0$, and hence (by Markov's inequality) that almost every ergodic summand of $\mu_0$ almost surely contains a conjugate of every Coxeter generator of $\Gamma$. In particular it must be spanning. 


\subsubsection{Faithfulness}

%
%
%
%

\begin{lem}\label{injrad_large}
  Let $C = 2+\sqrt 3+D$. For any $R>0$, for $n$ large enough and any $x \in P_n$ there is a point $y \in P_n$ such that $d(x, y) \le R+C$ and $B(y, R) \subset P_n$.

  We may further assume that $y \in \Gamma\cdot x$. 
\end{lem}

\begin{proof}
  Let $D$ be the diameter of $P$. Then $B(n) \subset P_n \subset B(n+C)$ by Lemma \ref{envelope} The first part of the result follows immediately: we project $x$ onto the convex subset $B(n)$ to get a point $x'$ at distance at most $2+\sqrt 3+D$ from $x$, and then move inside $B(n)$ along the radius through $x'$ until reaching a distance $\ge R$ from the boundary $\pl B(n)$: the resulting point satisfies the conclusions in the first paragraph of the statement of the lemma.

  The additional statement is immediate because $\Gamma$ is cocompact in $\HH^d$. 
\end{proof}

Because the set $\{g \in \Gamma : d_X(o, go) \le R+C\}$ is finite it follows from this lemma that for any $R > 0$ we have (denoting by $P_H$ the polyhedron associated with a reflection subgroup $H \in \sub_\Gamma$): 
\[
\PP_{\mu_1}(\exists g \in \Gamma :\: B(go, R) \subset P_H) = 1. 
\]
It follows that almost every ergodic summand of $\mu_1$ also satisfies this property. Let $\nu$ be such a summand. 

For a given $g \in \Gamma$ the subset of reflection subgroups $H \in \sub_\Gamma$ such that $B(gK, R) \subset P_H$ is an open subset in Chabauty topology. By compactness of the support of $\mu_1$ it follows that there exists $g_1, \ldots, g_N \in \Gamma$ such that  
\[
\PP_{\nu}(\exists 1\le i \le N :\: B(g_iK, R) \subset P_H) = 1. 
\]
so that there exists a $g \in \Gamma$ such that
\[
\PP_{\nu}(B(gK, R) \subset P_H) > 0.  
\]
Now $B(gK, R) \subset P_H$ means that for every nontrivial $x \in B_R$ we have $x \not\in gHg^{-1}$. By invariance of $\mu_1$ we get that
\[
\PP_{\nu}(H \cap B_R = \{1\}) = \PP_\nu(gHg^{-1} \cap B_R) > 0. 
\]
Since $R$ was arbitrary and $\Gamma = \bigcup_{R>0} B_R$ we get that for any nontrivial $x \in \Gamma$ we have $\PP_{\nu}(x \not\in H) > 0$, hence $\nu$ is faithful.


\bibliographystyle{alpha}
\bibliography{bib}

\end{document}